\documentclass[12pt,leqno]{article}

\usepackage{amsmath,amssymb,amsthm}
\usepackage[all]{xy}
\usepackage{lscape}

\makeatletter

\newtheorem{theorem}{Theorem}[section]

\newtheorem{lemma}[theorem]{Lemma}
\newtheorem{corollary}[theorem]{Corollary}

\theoremstyle{definition}

\newtheorem{example}[theorem]{Example}

\theoremstyle{remark}

\theoremstyle{remark}

\def\({{\rm (}}
\def\){{\rm )}}

\let\Mathrm\operator@font
\let\Cal\mathcal
\let\Bbb\mathbb
\newcommand{\fm}{\ensuremath{\mathfrak m}}
\newcommand{\fn}{\ensuremath{\mathfrak n}}

\def\standop#1{\mathop{\Mathrm #1}\nolimits}
\def\difstop#1#2{\expandafter\def\csname #1\endcsname{\standop{#2}}}
\def\defstop#1{\difstop{#1}{#1}}

\defstop{AB}
\defstop{ann}
\defstop{Ass}
\defstop{Add}
\defstop{Alt}
\defstop{Ass}
\defstop{add}
\defstop{asn}
\def\abs#1{\lvert{#1}\rvert}
\difstop{adim}{AFHdim}

\defstop{CMFI}
\defstop{codim}
\defstop{Coh}
\defstop{coht}
\defstop{Coker}
\defstop{Cone}
\defstop{Cl}
\defstop{Cox}
\defstop{cosk}
\defstop{cd}
\defstop{cmd}
\difstop{charac}{char}

\defstop{depth}
\defstop{Der}
\defstop{Div}
\defstop{div}
\defstop{det}

\defstop{EM}
\defstop{embdim}
\defstop{End}
\defstop{ev}
\defstop{Ext}

\difstop{fdim}{flat.dim}
\defstop{Flat}
\defstop{Func}
\defstop{Fpqc}
\defstop{FL}
\defstop{FS}
\def\floor#1{\lfloor{#1}\rfloor}

\def\GL{\text{\sl{GL}}}
\defstop{GD}
\defstop{Good}
\defstop{Gal}

\difstop{height}{ht}
\defstop{Hom}

\difstop{Image}{Im}
\defstop{Ind}
\defstop{ind}

\defstop{Ker}

\defstop{Lch}
\defstop{length}
\defstop{Lin}
\defstop{Lqc}
\defstop{lqc}
\defstop{LQ}
\defstop{LM}
\defstop{Lie}
\defstop{Loc}

\defstop{Mat}
\defstop{Max}
\defstop{Min}
\defstop{Mod}
\defstop{Mor}
\defstop{MCM}
\defstop{Map}
\difstop{mred}{red}
\difstop{md}{mod}

\defstop{Nerve}
\defstop{NonSFR}
\defstop{NonCMFI}
\defstop{NonNor}
\defstop{Nor}
\defstop{NF}
\defstop{nsurj}
\def\norm#1{\lVert{#1}\rVert}

\defstop{ob}
\defstop{Ob}

\defstop{PA}
\defstop{PM}
\defstop{PR}
\defstop{Proj}
\defstop{Prin}
\defstop{Pic}

\defstop{Qch}
\defstop{qch}

\defstop{rad}
\defstop{rank}

\defstop{res}
\defstop{Reg}
\defstop{Ref}
\def\RR{{\mathbb R}}

\defstop{Spec}
\defstop{supp}
\defstop{Supp}
\defstop{surj}
\defstop{Sym}
\defstop{Sing}
\defstop{SFR}
\defstop{Soc}
\defstop{soc}
\defstop{Sh}
\difstop{sm}{sum}

\difstop{tdeg}{trans.deg}

\defstop{Tor}
\defstop{TRD}
\difstop{trace}{tr}

\defstop{Zar}
\def\ZZ{{\mathbb Z}}

\def\C{\Cal C}

\def\fd{\mathfrak{d}}

\def\fm{\mathfrak{m}}

\def\sdarrow#1{\downarrow\hbox to 0pt{\scriptsize$#1$\hss}}
\def\suarrow#1{\uparrow\hbox to 0pt{\scriptsize$#1$\hss}}
\def\ssearrow#1{\searrow\hbox to 0pt{\scriptsize$#1$\hss}}

\def\ext{{\textstyle\bigwedge}}

\def\section{\@startsection{section}{1}{\z@ }%
  {-3.5ex plus -1ex minus -.2ex}{2.3ex plus .2ex}{\bf }}

\long\def\refname{\par\kern -3ex
  \begin{center}\rm R\sc{eferences}\end{center}\par\kern 
  -2ex}

\def\@seccntformat#1{\csname the#1\endcsname.\quad}

\def\@@@sect#1#2#3#4#5#6[#7]#8{%
  \ifnum #2>\c@secnumdepth 
  \def \@svsec {}\else \refstepcounter {#1}%
  \def\@svsec{}
  \fi 
  \@tempskipa #5\relax 
  \ifdim \@tempskipa >\z@ 
  \begingroup #6\relax \@hangfrom {\hskip #3\relax 
    \@svsec}{\interlinepenalty \@M #8\par }\endgroup 
  \csname #1mark\endcsname {#7}
  \else 
  \def \@svsechd {#6\hskip #3\@svsec #8\csname #1mark\endcsname {#7}}
  \fi \@xsect {#5}}

\def\@@@startsection#1#2#3#4#5#6{%
  \if@noskipsec \leavevmode \fi \par \@tempskipa #4\relax \@afterindenttrue 
  \ifdim \@tempskipa <\z@ \@tempskipa -\@tempskipa \@afterindentfalse 
  \fi \if@nobreak \everypar {}\else \addpenalty {\@secpenalty }\addvspace 
  {\@tempskipa }\fi \@ifstar {\@ssect {#3}{#4}{#5}{#6}}{\@dblarg 
    {\@@@sect {#1}{#2}{#3}{#4}{#5}{#6}}}}

\def\theparagraph{\thesection.\arabic{paragraph}}
\def\aparagraph{\@@@startsection{paragraph}{2}{\z@ }%
  {1.75ex plus .2ex minus .15ex}{-1em}{\bf(\theparagraph) } }
\def\paragraph{\@@@startsection{paragraph}{2}{\z@ }%
  {1.75ex plus .2ex minus .15ex}{-1em}{}{\bf(\theparagraph)} }

\let\c@theorem\c@paragraph

\title{$F$-rationality of the ring of modular invariants}
  \author{M{\sc itsuyasu} H{\sc ashimoto}}
\date{\normalsize
  Okayama University\\
  Okayama 700--8530, JAPAN\\
  {\small \tt mh@okayama-u.ac.jp}\\
  ~\\
  Dedicated to Professor Kei-ichi Watanabe
}

\begin{document}

\maketitle

\begin{abstract}
  Using the description of the Frobenius
  limit of modules over the ring
  of invariants under an action
  of a finite group on a polynomial ring
  over a field of characteristic $p>0$ developed by Symonds and
  the author,
  we give a characterization of
  the ring of invariants with a positive dual $F$-signature.
  Combining this result and Kemper's result on depths of the ring of
  invariants under
  an action of a permutation group, we give an example of an $F$-rational,
  but non-$F$-regular ring of invariants under the action of a finite group.
\end{abstract}

\section{Introduction}

Let $k$ be an algebraically closed field of characteristic $p>0$.
Let $V=k^d$, and $G$ a finite subgroup of $\GL(V)$ without psuedo-reflections.
Let $B=\Sym V$, the symmetric algebra of $V$, and $A=B^G$.
If the order $\abs{G}$ of $G$
is not divisible by $p$, then $A$ is a direct summand subring
of $B$, and is strongly $F$-regular.

Let $p$ divides $\abs{G}$.
Broer \cite{Broer} proved that 
$A$ is not a direct summand subring of $B$ hence $A$ is not weakly $F$-regular
(as $A$ is not a splinter).

In this paper, we study when $A$ is $F$-rational.
In \cite{Glassbrenner}, Glassbrenner showed that the invariant subring
$k[x_1,\ldots,x_n]^{A_n}$, where $A_n$ is the alternating group which acts on
the polynomial ring by permutation of variables, is Gorenstein $F$-pure but
not $F$-rational when $p=\charac(k)\geq 3$ and $n\equiv 0,1 \mod p$.
In this paper, we give an example of $A$ which is $F$-rational
(but not $F$-regular).

Sannai \cite{Sannai} defined the dual $F$-signature $s(M)$ of a finite module
$M$ over an $F$-finite
local ring $R$ of characteristic $p$.
He proved that $R$ is $F$-rational if and only if $R$ is Cohen--Macaulay and
the dual $F$-signature $s(\omega_R)$ of the canonical module $\omega_R$ of $R$
is positive.
Utilizing the description of the Frobenius limit of modules over
$\hat A$ (the completion of $A$) by Symonds and the author,
we give a characterization of $V$ such that $s(\omega_{\hat A})>0$,
see Theorem~\ref{main.thm}.
The characterization is purely representation theoretic in the sense that the
characterization depends only on the structure of $B$ as a $G$-module,
rather than a $G$-algebra.

Using the characterization and Kemper's result on the depth of the ring of invariants
under the action of certain groups of permutations \cite[(3.3)]{Kemper},
we give an example of $F$-rational $A$ for $p\geq 5$.
We also get an example of $A$ such that the dual $F$-signature $s_{\omega_{\hat A}}$ of
the canonical module of the completion $\hat A$ is positive, but $A$ (or equivalently,
$\hat A$) is not Cohen--Macaulay.
See Theorem~\ref{main-example.thm}.

In section~\ref{asn.sec}, we introduce the
{\em asymptotic surjective number} $\asn_N(M)$ for two finitely
generated modules $M$ and $N$ ($N\neq 0$) over a Noetherian ring $R$,
see Lemma~\ref{asn.lem}.
In section~\ref{dual-F-signature.sec}, using the definition and some basic results
developed in section~\ref{asn.sec}, we prove the formula
$s(M)=\asn_M(\FL([M]))$, where $\FL$ denotes the Frobenius
limit defined in \cite{SH}.
Thus $s(M)$ depends only on $\FL([M])$.
Using this, we give a characterization of a module $M$ to have
positive $s(M)$ in terms of $\FL([M])$ (Corollary~\ref{s-positive.cor}).

Using this result and the description of the Frobenius
limits of certain modules over
$\hat A$ proved in \cite{SH}, we give a characterization of $V$ such that
$s(\omega_{\hat A})>0$ in section~\ref{main.sec}.

In section~\ref{main-example.sec}, we give the examples.

\medskip

Acknowledgments.
The author is grateful to
Professor Anurag Singh
and
Professor Kei-ichi Watanabe
for valuable discussion.

\section{Asymptotic surjective number}\label{asn.sec}

\paragraph
This paper heavily depends on \cite{SH}.

\paragraph
Let $R$ be a Noetherian commutative ring.
Let $\md R$ denote the category of finite $R$-modules.

\paragraph
For $M,N\in\md R$, we set
\begin{multline*}
\surj_N^R(M)=
\surj_N(M):=
\\
\sup\{n\in\Bbb Z_{\geq 0}\mid \text{There is a surjective $R$-linear
  map $M\rightarrow N^{\oplus n}$}\},
\end{multline*}
and call $\surj_N(M)$ the surjective number of $M$ with respect to $N$.
If $N=0$, this is understood to be $\infty$.

\begin{lemma}\label{surj.lem}
  Let $M,M',N\in\md R$.
  Then we have the following.
  \begin{enumerate}
  \item[\bf 1] If $R'$ is any Noetherian $R$-algebra, then
    \[
    \surj_N^R(M)\leq \surj_{R'\otimes_R N}^{R'}(R'\otimes_R M).
    \]
  \item[\bf 2] If $(R,\fm)$ is local and $N\neq 0$, then
    $\surj^R_N(M)\leq \mu_R(M)/\mu_R(N)$, where $\mu_R=\ell_R(R/\fm\otimes_R ?)$ denotes
    the number of generators.
  \item[\bf 3] If $N\neq 0$, then $\surj_N(M)<\infty$, and is a non-negative integer.
  \item[\bf 4] If $N\neq 0$, then $\surj_N(M)+\surj_N(M')\leq \surj_N(M\oplus M')$.
  \item[\bf 5] If $N\neq 0$ and $r\geq 0$, then $r\surj_N(M)\leq
    \surj_N(M^{\oplus r})$.
    \end{enumerate}
\end{lemma}

\begin{proof}
  {\bf 1}
  If there is a surjective $R$-linear map $M\rightarrow N^{\oplus n}$,
  then there is a surjective $R'$-linear map $R'\otimes_R M
  \rightarrow (R'\otimes_R N)^{\oplus n}$, and hence $n\leq \surj_{R'\otimes_R N}^{R'}(
  R'\otimes_R M)$.

  {\bf 2} By {\bf 1}, $\surj^R_N(M)\leq \surj^{R/\fm}_{N/\fm N}(M/\fm M)\leq
  \mu_R(M)/\mu_R(N)$ by dimension counting.

  {\bf 3} Take $\fm\in \supp_R N$.
  Then
\[
  \surj^R_N(M)\leq \surj^{R_\fm}_{N_\fm}(M_\fm)\leq \mu_{R_\fm}(M_\fm)/\mu_{R_\fm}(N_\fm)
  <\infty.
\]

  {\bf 4} Let $n=\surj_N(M)$ and $n'=\surj_N(M')$.
  Then there are surjective $R$-linear maps $M\rightarrow N^{\oplus n}$ and
  $M'\rightarrow N^{\oplus n'}$.
  Summing them, we get a surjective map
  $M\oplus M'\rightarrow N^{\oplus(n+n')}$.

  {\bf 5} follows from {\bf 4}.
\end{proof}

\paragraph
Let $N,M\in\md R$.
Assume that $N$ is nonzero.
We define
\[
\nsurj_N(M;r):=\frac{1}{r}\surj_N(M^{\oplus r})
\]
for $r\geq 1$.

\begin{lemma}\label{nsurj.lem}
  Let $r\geq 1$, and $M,M',N\in\md R$.
  Assume that $N\neq 0$.
  Then
  \begin{enumerate}
  \item[\bf 1] $\nsurj_N(M;1)=\surj_N(M)$.
  \item[\bf 2] $\nsurj_N(M;kr)\geq \nsurj_N(M;r)$ for $k\geq 0$.
  \item[\bf 3] $\nsurj_N(M;r)\geq \surj_N(M)\geq 0$.
  \item[\bf 4] $\nsurj_N(M;r)+\nsurj_N(M';r)\leq \nsurj_N(M\oplus M';r)$.
  \item[\bf 5] If $R\rightarrow R'$ is a homomorphism of Noetherian rings, then
    $\nsurj_N(M;r)\leq \nsurj_{R'\otimes_R N}(R'\otimes_R M;r)$.
  \item[\bf 6] 
    If $(R,\fm)$ is local, $\nsurj_N(M;r)\leq \mu_R(M)/\mu_R(N)$.
    In general, $\nsurj_N(M;r)$ is bounded.
  \end{enumerate}
\end{lemma}

\begin{proof}
  {\bf 1} is by definition.

  {\bf 2}. $kr\nsurj_N(M;kr)=\surj_N(M^{\oplus kr})\geq k\surj_N(M^{\oplus r})$
  by Lemma~\ref{surj.lem}, {\bf 5}.
  Dividing by $kr$, we get the desired inequality.

  {\bf 3}. This is immediate by {\bf 1} and {\bf 2}.
  
  {\bf 4} follows from Lemma~\ref{surj.lem}, {\bf 4}.
  
  {\bf 5} follows from Lemma~\ref{surj.lem}, {\bf 1}.

  {\bf 6} The first assertion is by Lemma~\ref{surj.lem}, {\bf 2}.
  The second assertion follows from the first assertion and {\bf 5} applied to
  $R\rightarrow R'=R_\fm$, where $\fm$ is any element of $\supp_R N$.
\end{proof}

\begin{lemma}\label{asn.lem}
  Let $M,N\in\md R$.
  Assume that $N\neq 0$.
  Then the limit
\[
\lim_{r\rightarrow\infty}\nsurj_N(M;r)
=
\lim_{r\rightarrow\infty}\frac{1}{r}\surj_N(M^{\oplus r})
\]
exists.
\end{lemma}

We call the limit the {\em
  asymptotic surjective number of $M$ with respect to $N$},
and denote it by $\asn_N(M)$.

\begin{proof}
  As $\nsurj_N(M;r)$ is bounded,
  $S=\limsup_{r\rightarrow\infty}\nsurj_N(M;r)$ and
  $I=\liminf_{r\rightarrow\infty}\nsurj_N(M;r)$ exist.
  Assume for contradiction that the limit does not exist.
  Then $S>I$.
  Set $\varepsilon=(S-I)/2>0$.

  There exists some $r_0\geq 1$ such that $\nsurj_N(M;r_0)>S-\varepsilon/2$.
  Take $n_0\geq 1$ sufficiently large so that
  $\nsurj_N(M;r_0)/n_0<\varepsilon/2$.
  Let $r\geq r_0n_0$, and set $n:=\floor{r/r_0}$.
  Note that $nr_0\leq r<(n+1)r_0$ and $n\geq n_0$.

  Then
  \begin{multline*}
  \nsurj_N(M;r)\geq \frac{1}{(n+1)r_0}\surj_N(M^{\oplus nr_0})
  \geq \frac{n}{(n+1)r_0}\surj_N(M^{\oplus r_0})\\
  =(1-\frac{1}{n+1})\nsurj_N(M;r_0)\geq \nsurj_N(M;r_0)-\varepsilon/2
  >S-\varepsilon.
  \end{multline*}
  Hence
  \[
  I\geq \inf_{r\geq r_0n_0}\nsurj_N(M;r)\geq S-\varepsilon>S-2\varepsilon=I,
  \]
  and this is a contradiction.
\end{proof}

\begin{lemma}
  Let $M,M',N\in\md R$, and $N\neq 0$.
  Then
  \begin{enumerate}
  \item[\bf 1] $\asn_N(M^{\oplus r})=r\asn_N(M)$.
  \item[\bf 2] $0\leq \surj_N(M)\leq\nsurj_N(M;r)\leq \asn_N(M)$ for any $r\geq 1$.
  \item[\bf 3] $\asn_N(M)+\asn_N(M')\leq \asn_N(M\oplus M')$.
  \end{enumerate}
\end{lemma}

\begin{proof}
  {\bf 1}.
  \[
  r^{-1}\asn_N(M^{\oplus r})=\lim_{r'\rightarrow \infty}\frac{1}{rr'}\surj_N(M^{\oplus rr'})
  =\asn_N(M).
  \]

  {\bf 2}. $0\leq \surj_N(M)\leq \nsurj_N(M;r)$ is Lemma~\ref{nsurj.lem}, {\bf 3}.
  So taking the limit, $\surj_N(M)\leq \asn_N(M)$.
  So $\surj_N(M^{\oplus r})\leq \asn_N(M^{\oplus r})=r\asn_N(M)$.
  Dividing by $r$, $\nsurj_N(M;r)\leq \asn_N(M)$.
\end{proof}

\begin{lemma}\label{easy-vector-space.lem}
  Let $k$ be a field, and $V$ a $k$-vector space, and $n\geq 0$.
  Assume that $\dim_k V\leq n$.
  Let $\Gamma$ be a set of subspaces of $V$ such that $\sum_{U\in\Gamma}U=V$.
  Then there exist some $U_1,\ldots,U_{n'}\in\Gamma$ with $n'\leq n$
  such that $U_1+\cdots+U_{n'}=V$.
\end{lemma}

\begin{proof}
  Trivial.
\end{proof}

\begin{lemma}\label{easy-vector-space2.lem}
  Let $k$ be a field, $V$ a $k$-vector space, and $\Gamma$ a set of subspaces of
  $V$.
  Let $W$ and $W'$ be subspaces of $V$ such that $W+W'=V$.
  Assume that $W'\subset \sum_{U\in\Gamma}U$.
  If $\dim_k W'\leq n$, then there exist some $U_1,\ldots,U_{n'}\in\Gamma$ with
  $n'\leq n$ such that 
  $W+U_1+\cdots+U_{n'}=V$.
\end{lemma}

\begin{proof}
  Apply Lemma~\ref{easy-vector-space.lem} to the vector space $V/W$.
\end{proof}

\begin{lemma}\label{surj-mu.lem}
  Let $(R,\fm)$ be a Noetherian local ring.
  Let $M,M',N\in\md R$ with $N\neq 0$.
  Then
  \[
  \surj_N(M')\leq \surj_N(M\oplus M')-\surj_N(M)\leq \mu_R(M').
  \]
\end{lemma}

\begin{proof}
  The first inequality is Lemma~\ref{surj.lem}, {\bf 4}.
  We prove the second inequality.
  Let $m=\surj_N(M\oplus M')$ and $n=\mu_R(M')$.
  There is a surjective map $\varphi:M\oplus M'\rightarrow N^{\oplus m}$.
  Let $N_i=N$ be the $i$th summand of $N^{\oplus m}$.
  Let $\bar ?$ denote the functor $R/\fm$.
  Set $V=\bar N^{\oplus m}$, $W=\bar\varphi(\bar M)$, and $W'=\bar\varphi(\bar M')$.
  Then by Lemma~\ref{easy-vector-space2.lem}, there exists some
  index set $I\subset \{1,2,\ldots,m\}$ such that $\#I\leq n$ and
  $W+\sum_{i\in I}\bar N_i=V$.
  By Nakayama's lemma, $\varphi(M)+\sum_{i\in I}N_i=N^{\oplus m}$.
  This shows that
  \[
  M\hookrightarrow M\oplus M'\xrightarrow \varphi
  N^{\oplus m}\rightarrow N^{\oplus m}/\sum_{i\in I}N_i\cong N^{\oplus (m-\#I)}
  \]
  is surjective.
  Hence $\surj_N(M)\geq m-\#I \geq m-n$, and the result follows.
\end{proof}

\paragraph
Let $(R,\fm)$ be a Henselian local ring.
Let $\C:=\md R$.
As in \cite{SH}, we define
\[
[\C]:=(\bigoplus_{M\in\C}\ZZ\cdot M)
/(M-M_1-M_2\mid M\cong M_1\oplus M_2),
\]
and $[\C]_\RR:=\RR\otimes_\ZZ[\C]$.
In \cite{SH},
$[\C]_\RR$ is also written as $\Theta^{\wedge}(R)$ or $\Theta(R)$ (considering that
$R$ is trivially graded).
In this paper, we write it as $\Theta(R)$.
For $M\in\C$, we denote by $[M]$ the class of $M$ in $\Theta(R)$.
For 
an isomorphism class $N$ of modules, $[N]$ is a well-defined element of $\Theta(R)$.
Let $\Ind(R)$ denote the set of isomorphism classes of indecomposable modules in $\C$.
The set $[\Ind(R)]:=\{[M]\mid M\in\Ind(R)\}$ is an $\RR$-basis of $\Theta(R)=[\C]_\RR$.
So $\alpha\in\Theta(R)$ can be written $\alpha=\sum_{M\in\Ind(R)}c_M[M]$ with $c_M\in\RR$ uniquely.
We say that $\alpha\geq 0$ if $c_M\geq 0$ for any $M\in\Ind(R)$.
For $\alpha,\beta\in\Theta(R)$, we define $\alpha\geq \beta$ if $\alpha-\beta\geq 0$.
This gives an ordering on $\Theta(R)$.

\paragraph
For $\alpha=\sum_{M\in\Ind(R)}c_M[M]\in\Theta(R)$, we define
\[
\langle \alpha\rangle:=\sum_{M\in\Ind(R)}\max(0,\floor{c_M})[M].
\]
So there exists some $M_\alpha\in\C$, unique up to isomorphisms, such that $\langle\alpha\rangle=[M_\alpha]$.
For $N\in\md R$ with $N\neq 0$, we define $\surj_N\alpha$ to be $\surj_N M_\alpha$.

\paragraph
For $\alpha=\sum_{M\in\Ind(R)}c_M M
\in\Theta(R)$, we define $\supp \alpha=\{M\in\Ind(R)\mid c_M>0\}$.
We define $Y(\alpha)=\bigoplus_{W\in\supp\alpha}W$ and $\nu(\alpha):=\mu_R(Y(\alpha))$.

\begin{lemma}\label{alpha-surj.lem}
  Let $N\in\md R$, $N\neq 0$, and $\alpha,\beta \in \Theta(R)$.
  \begin{enumerate}
  \item[\bf 1] If $\alpha,\beta\geq 0$, then
    $0\leq \surj_N\alpha\leq \surj_N(\alpha+\beta)-\surj_N\beta$.
  \item[\bf 2]
   $\abs{\surj_N\alpha-\surj_N\beta}\leq \norm{\alpha-\beta}+\nu(\inf\{\alpha,\beta\})$.
  \end{enumerate}
\end{lemma}

\begin{proof}
  {\bf 1}. As $\alpha,\beta\geq 0$, we have that
  $\langle \alpha\rangle+\langle \beta\rangle\leq\langle\alpha+\beta\rangle$.
  So by Lemma~\ref{surj.lem}, {\bf 4}, 
  $\surj_N\alpha+\surj_N\beta \leq \surj_N(\langle\alpha+\beta\rangle)\leq
  \surj_N(\alpha+\beta)$.

  {\bf 2}.
  Replacing $\alpha$ by $\sup\{\alpha,0\}$ and $\beta$ by $\sup\{\beta,0\}$,
  we may assume that $\alpha,\beta\geq 0$.
  Moreover, replacing $\alpha$ by $\sup\{\alpha,\beta\}$ and $\beta$ by
  $\inf\{\alpha,\beta\}$, we may assume that $\alpha\geq \beta$.
  As we have
  $\langle\alpha\rangle-\langle\beta\rangle\leq \alpha-\beta+[Y(\beta)]$,
  by Lemma~\ref{surj-mu.lem} we have that
  \begin{multline*}
  \surj_N\alpha-\surj_N\beta\leq\norm{\langle\alpha\rangle-\langle\beta\rangle}
  \leq
  \norm{\alpha-\beta+[Y(\beta)]}\\
  \leq\norm{\alpha-\beta}+\norm{[Y(\beta)]}
  =\norm{\alpha-\beta}+\nu(\beta).
  \end{multline*}
  This is what we wanted to prove.
  \end{proof}

\begin{lemma}
The limit 
  \[
  \lim_{t\rightarrow \infty}\frac{1}{t}\surj_N(t\alpha)
  \]
  exists for $N\in\md R$, $N\neq 0$ and $\alpha\in\Theta(R)$.
\end{lemma}

We denote the limit by $\asn_N(\alpha)$.

\begin{proof}
  Replacing $\alpha$ by $\sup\{0,\alpha\}$, we may assume that $\alpha\geq 0$.
  Let $\varepsilon>0$.
  We can take $W\in\md R$ and an integer $n>0$ such that $\alpha-n^{-1}[W]\geq 0$ and 
  $\norm{\alpha-n^{-1}[W]}<\varepsilon/8$.
  As $\asn_N W$ exists, there exists some $r_0\geq 1$ such that for any $r\geq r_0$,
  $\abs{\nsurj_N(W;r)-\asn_N W}<n\varepsilon/8$.
  Set $R:=\max\{r_0n,16\mu_R(W)/\varepsilon,8n\norm{\alpha}/\varepsilon\}$.
  Let $t>R$.
  Let $r:=\floor{t/n}$.
  Then $0\leq t-rn < n$ and $r\geq r_0$.
  We have
  \begin{multline*}
    \abs{t^{-1}\surj_N(t\alpha)- n^{-1}\asn_N W}
      \leq
  t^{-1}\abs{\surj_N(t\alpha)-\surj_N(W^{\oplus r})}\\
  +((rn)^{-1}-t^{-1})\surj_N(W^{\oplus r})
  +\abs{(rn)^{-1}\surj_N(W^{\oplus r})-n^{-1}\asn_N W}
  \\
  < t^{-1}\norm{t\alpha-r[W]}
  +
  t^{-1}\mu_R(W)
    +
  (rt)^{-1}\mu_R(W^{\oplus r})
    +
   \varepsilon/8
   \\
   \leq 
   (n/t)\norm{\alpha}+(nr/t)\norm{\alpha-n^{-1}[W]}+\varepsilon/16+\varepsilon/16+\varepsilon/8
   \\
   <\varepsilon/8+\varepsilon/8+\varepsilon/16+\varepsilon/16+\varepsilon/8=\varepsilon/2.
  \end{multline*}
  So for $t_1,t_2>R$,
  \[
  \abs{t_1^{-1}\surj_N(t_1\alpha)-t_2^{-1}\surj_N(t_2\alpha)}<\varepsilon,
  \]
  and $\lim_{t\rightarrow\infty}t^{-1}\surj_N(t\alpha)$ exists, as desired.
\end{proof}

\begin{lemma}
  Let $\alpha,\beta\in\Theta(R)$ and $N\in\md R$ with $N\neq 0$.
  \begin{enumerate}
  \item[\bf 1] For $k\geq 0$, we have $\asn_N(k\alpha)=k\asn_N(\alpha)$.
  \item[\bf 2] For $k\geq 0$,
    $0\leq \surj_N(k\alpha)\leq k\asn_N(\alpha)\leq k\norm{\alpha}/\mu_R(N)$.
  \item[\bf 3] If $\alpha,\beta\geq 0$, then
    $\asn_N(\alpha+\beta)\geq \asn_N(\alpha)+\asn_N(\beta)$.
  \item[\bf 4] $\abs{\asn_N(\alpha)-\asn_N(\beta)}\leq \norm{\alpha-\beta}$.
    \item[\bf 5] $\asn_N$ is continuous.
  \end{enumerate}
\end{lemma}

  \begin{proof}
    {\bf 1}.
    If $k=0$, then both-hand sides are zero, and the assertion is clear.
    So we may assume that $k>0$.
    Then
    \[
    \asn_N(k\alpha)=\lim_{t\rightarrow\infty}\frac{1}{t}\surj(tk\alpha)
    =k\lim_{t\rightarrow\infty}\frac{1}{tk}\surj(tk\alpha)=k\asn_N(\alpha).
    \]

    {\bf 2}.
    We may assume that $k>0$.
    By {\bf 1}, replacing $k\alpha$ by $\alpha$, we may assume that $k=1$.
    Replacing $\alpha$ by $\sup\{0,\alpha\}$, we may assume that $\alpha\geq 0$.
    For $n\geq 0$, $n\langle \alpha\rangle \leq \langle n\alpha\rangle$.
    Hence, $n\surj_N(\alpha)\leq \surj_N(n\langle\alpha\rangle)
    \leq\surj_N(n\alpha)$.
    So $\surj_N(\alpha)\leq n^{-1}\surj_N(n\alpha)$.
    Passing to the limit, $\surj_N(\alpha)\leq \asn_N(\alpha)$.
    Similarly,
    \[
    \frac{1}{n}\surj_N(n\alpha)\leq \frac{\norm{\langle n\alpha\rangle}}{n\mu_R(N)}
    \leq \frac{\norm{n\alpha}}{n\mu_R(N)}=\frac{\norm{\alpha}}{\mu_R(N)}.
    \]
    Passing to the limit,
    $\asn_N(\alpha)\leq \frac{\norm{\alpha}}{\mu_R(N)}$, as desired.

    {\bf 3}.
    By Lemma~\ref{alpha-surj.lem}, {\bf 1},
    for $t>0$,
    \[
    \frac{1}{t}\surj_N(t\alpha)+\frac{1}{t}\surj_N(t\beta)
    \leq
    \frac{1}{t}\surj_N(t(\alpha+\beta)).
    \]
    Passing to the limit, $\asn_N(\alpha)+\asn_N(\beta)\leq \asn_N(\alpha+\beta)$.

    {\bf 4}.
    By Lemma~\ref{alpha-surj.lem}, {\bf 2},
    \begin{multline*}
    \abs{\frac{1}{t}\surj_N(t\alpha)-\frac{1}{t}\surj_N(t\beta)}
    \leq \frac{1}{t}(\norm{t(\alpha-\beta)}+\nu(\inf\{t\alpha,t\beta\}))\\
    = \norm{\alpha-\beta}+\nu(\inf\{\alpha,\beta\})/t.
    \end{multline*}
    Passing to the limit, $\abs{\asn_N(\alpha)-\asn_N(\beta)}\leq \norm{\alpha-\beta}$,
    as desired.

    {\bf 5} is an immediate consequence of {\bf 4}.
  \end{proof}

  \section{Sannai's dual $F$-signature}\label{dual-F-signature.sec}

  \paragraph
  In this section, let $p$ be a prime number, and
  $(R,\fm,k)$ be an $F$-finite local ring of characteristic $p$ of dimension $d$.
  Let $\fd =\log_p[k:k^p]$, and $\delta=d+\fd$.

  \paragraph
  In \cite{Sannai}, for $M\in\md R$, Sannai defined
  the dual $F$-signature of $M$ by
  \[
  s_R(M)=s(M):=\limsup_{e\rightarrow\infty}\frac{\surj_M({}^eM)}{p^{\delta e}}.
  \]

  $s(R)$ is the (usual) $F$-signature \cite{HL}, which is closely related to
  the strong $F$-regularity of $R$ \cite{AL}.
  While $s(\omega_R)$ measures the $F$-rationality of $R$, provided $R$ is
  Cohen--Macaulay.

  \begin{theorem}[{\cite[(3.16)]{Sannai}}]\label{Sannai.thm}
    $R$ is $F$-rational if and only if $R$ is Cohen--Macaulay and $s(\omega_R)>0$.
  \end{theorem}
  
  Now we connect the Frobenius
  limit defined in \cite{SH} with dual $F$-signature.

  \begin{theorem}
    Let $R$ be Henselian, and $M\in\md R$.
    Assume that the Frobenius limit
\[
\FL([M])=\lim_{e\rightarrow \infty}\frac{1}{p^{\delta e}}[{}^eM]\in\Theta(R)
\]
exists.
    Then
    \[
    s_R(M)=\lim_{e\rightarrow\infty}\frac{\surj_M({}^eM)}{p^{\delta e}}=\asn_M(\FL([M])).
    \]
  \end{theorem}

  \begin{proof}
    By Lemma~\ref{alpha-surj.lem}, 
    \begin{multline*}
      p^{-\delta e} \abs{\surj_M(p^{\delta e}\FL([M]))-\surj_M([{}^eM])}
      \\
      \leq
    \norm{\FL([M])-p^{-\delta e}[{}^eM]}+p^{-\delta e}\nu(\supp(\FL([M]))).
    \end{multline*}
    Taking the limit $e\rightarrow\infty$, we get the desired result.
  \end{proof}

  \begin{corollary}\label{s-positive.cor}
    Let the assumption be as in the theorem.
    Then the following are equivalent.
    \begin{enumerate}
    \item[\bf 1] $s(M)>0$.
    \item[\bf 2] For any $N\in\md R$ such that
      $\supp([N])=\supp(\FL(M))$, there exists some $r\geq 1$ and a surjective
      $R$-linear map $N^{\oplus r}\rightarrow M$.
    \item[\bf 3] There exist some $N\in\md R$ such that
      $\supp([N])\subset\supp(\FL(M))$ and a surjective $R$-linear map $N\rightarrow M$.
    \end{enumerate}
  \end{corollary}

  \begin{proof}
    {\bf 1$\Rightarrow$2}.
    As $\asn_M(\FL(M))>0$, there exists some $t>0$ such that $\surj_M(t\FL(M))>0$.
    By the choice of $N$, there exists some $r\geq 1$ such that
    $r[N]\geq t\FL(M)$ and so $\surj_M N^{\oplus r}\geq \surj_M(t\FL(M))>0$.

    {\bf 2$\Rightarrow$3}.
    Let $N=W_1\oplus\cdots\oplus W_r$, where $\{W_1,\ldots,W_r\}=\supp(\FL(M))$.
    Then there exists some $r\geq 1$ and a surjective $R$-linear map
    $N^{\oplus r}\rightarrow M$, and $\supp[N^{\oplus r}]\subset \supp(\FL(M))$.

    {\bf 3$\Rightarrow$1}.
    By the choice of $N$, there exists some $k>0$ such that 
    $k\FL(M)\geq [N]$.
    Then $s(M)=\asn_M(\FL(M))\geq k^{-1}\asn_M[N]\geq k^{-1}\surj_M[N]>0$.
  \end{proof}

  \section{The dual $F$-signature of the ring of invariants}\label{main.sec}

  Utilizing the result in \cite{SH} and the last section, we give a criterion for the
  condition $s(\omega_{\hat A})>0$ for the ring of invariants $A$, where $\hat A$ is
  the completion.

  \paragraph
  Let $k$ be an algebraically closed field, $V=k^d$, $G$ a finite subgroup of $\GL(V)$.
  In this section, we assume that $G$ does not have a pseudo-reflection,
  where we say that $g\in\GL(V)$ is a pseudo-reflection if $\rank(g-1_V)=1$.
  Let $v_1,\ldots,v_d$ be a fixed $k$-basis of $V$.
  Let $B:=\Sym V=k[v_1,\ldots,v_d]$, and $A=B^G$.
  Let $\fm$ and $\fn$ be the irrelevant ideals of $A$ and $B$, respectively.
  Let $\hat A$ and $\hat B$ be the completion of $A$ and $B$, respectively.

  For a $G$-module $W$, we define $M_W:=(B\otimes_k W)^G$.
  Let $k=V_0,V_1,\ldots,V_n$ be the irreducible representations of $G$.
  Let $P_i\rightarrow V_i$ be the projective cover.
  Set $M_i:=M_{P_i}=(B\otimes_k P_i)^G$.
  For a finite dimensional
  $G$-module $W$, $\det_W$ denote the determinant representation $\ext^{\dim W} W$ of $W$.
  Let $V_\nu=\det_V$ be the determinant representation of $V$.

  \begin{lemma}\label{canonical.lem}
    The canonical module $\omega_A$ of $A$ is isomorphic to $M_\nu=M_{\det_V}$.
  \end{lemma}

  \begin{proof}
    See \cite[(14.28)]{Hashimoto12} and references therein.
  \end{proof}

  \begin{lemma}\label{selfinjective.lem}
    Let $\Lambda$ be a selfinjective finite dimensional $k$-algebra, $L$ a simple
    \(left\) $\Lambda$-module, and $h:P\rightarrow L$ its projective cover.
    Let $M$ be a finitely generated indecomposable $\Lambda$-module.
    Then the following are equivalent.
    \begin{enumerate}
    \item[\bf 1] $\Ext^1_\Lambda(M,\rad P)=0$.
    \item[\bf 2] $h_*:\Hom_\Lambda(M,P)\rightarrow \Hom_\Lambda(M,L)$ is surjective.
    \item[\bf 3] $M$ is either projective, or $M/\rad M$ does not contain $L$.
    \end{enumerate}
  \end{lemma}

  \begin{proof}
    {\bf 1$\Leftrightarrow$2}.
    This is because
    \[
    \Hom_\Lambda(M,P)\xrightarrow{h_*}
    \Hom_\Lambda(M,L)\rightarrow
    \Ext^1_\Lambda(M,\rad P)\rightarrow
    \Ext^1_\lambda(M,P)
    \]
    is exact and $\Ext^1_\Lambda(M,P)=0$ (since $P$ is injective).

    {\bf 2$\Rightarrow$3}.
    Assume the contrary.
    Then as $M/\rad M$ contains $L$, there is a surjective map $M\rightarrow L$.
    By assumption, this map lifts to $M\rightarrow P$, and this is surjective by
    Nakayama's lemma.
    As $P$ is projective, this map splits.
    As $M$ is indecomposable, $M\cong P$, and this is a contradiction.

    {\bf 3$\Rightarrow$2}.
    If $M$ is projective, then $h_*$ is obviously surjective.
    If $M/\rad M$ does not contain $L$, then $\Hom_\Lambda(M,L)=0$, and
    $h_*$ is obviously surjective.
  \end{proof}

  \begin{theorem}\label{main.thm}
    Let $p$ divide the order $\abs{G}$ of $G$.
    Then the following are equivalent.
    \begin{enumerate}
    \item[\bf 1] $s(\omega_{\hat A})>0$.
    \item[\bf 2] The canonical map $M_\nu\rightarrow M_{V_\nu}=\omega_A$ is surjective.
    \item[\bf 3] $H^1(G,B\otimes_k \rad P_\nu)=0$.
    \item[\bf 4] For any non-projective finitely generated
      indecomposable $G$-summand $M$ of $B$,
      $M$ does not contain $\det^{-1}_V$, the $k$-dual of $\det_V$.
    \end{enumerate}
If these conditions hold, then $s(\omega_{\hat A})\geq 1/\abs{G}$.
  \end{theorem}
  
  \begin{proof}
    We prove the equivalence of {\bf 2} and {\bf 3} first.
    Let $B=\bigoplus_j N_j$ be a decomposition into finitely generated indecomposable
    $G$-modules.
    Such a decomposition exists, since $B$ is a direct sum of finitely generated
    $G$-modules.
    The map $M_\nu\rightarrow M_{V_\nu}$ in {\bf 2} is the map
    \[
    (B\otimes P_\nu)^G\rightarrow (B\otimes \det_V)^G
    \]
    induced by the projective cover $P_\nu\rightarrow \det_V$.
    By the isomorphism $\Ext^i_G(N_j^*,?)\cong H^i(G,N_j\otimes ?)$, this map can
    be identified with the sum of 
    \[
    \Hom_G(N_j^*,P_\nu)\rightarrow \Hom_G(N_j^*,\det_V).
    \]
    On the other hand, {\bf 3} is equivalent to say that $\Ext^1_G(N_j^*,\rad P_\nu)=0$
    for any $j$.
    So the equivalence {\bf 2$\Leftrightarrow$3}
    follows from Lemma~\ref{selfinjective.lem}.

    Similarly, {\bf 4} is equivalent to say that each $N_j^*$ is injective
    (or equivalently, projective, as $kG$ is selfinjective) or
    $N_j^*/\rad N_j^*\cong (\soc N_j)^*$ does not contain $\det_V$.
    This is equivalent to say that $N_j$ is either projective, or $N_j$ (or equivalently,
    $\soc N_j$) does not contain $\det^{-1}$.
    So {\bf 4$\Leftrightarrow$2} follows from Lemma~\ref{selfinjective.lem}.

    We prove {\bf 2$\Rightarrow$1}.
    As there is a surjective map $M_\nu\rightarrow \omega_A$ and
    \[
    \FL([\omega_{\hat A}])=\frac{1}{|G|}\sum_{i=0}^n(\dim V_i)[\hat M_i]
    \]
    by \cite[(5.1)]{SH}, $s(\omega_{\hat A})>0$ by
    Corollary~\ref{s-positive.cor}.
    Moreover,
    \[
    s(\omega_{\hat A})=\asn_{\omega_{\hat A}}(\FL([\omega_{\hat A}]))
    \geq \frac{\dim V_\nu}{\abs{G}}\asn_{\omega_{\hat A}}(\hat M_\nu)
    \geq \frac{1}{\abs{G}}\surj_{\omega_A}(M_\nu)\geq \frac{1}{\abs{G}},
    \]
    and the last assertion has been proved.

    We prove {\bf 1$\Rightarrow$2}.
    By \cite[(4.16)]{SH},
    \[
    \FL([\omega_{\hat A}])=\frac{1}{\abs{G}}[\hat B].
    \]
    So by Corollary~\ref{s-positive.cor},
    there is some $r>0$ and a surjective map $h:\hat B^r\rightarrow \omega_{\hat A}$.
    By the equivalence $\gamma=(\hat B\otimes_{\hat A}?)^{**}:
    \Ref(\hat A)\rightarrow \Ref(G,\hat B)$ (see \cite[(2.4)]{HN} and
    \cite[(5.4)]{SH}), there corresponds
    \[
    \tilde h=\gamma(h): (\hat B\otimes_k kG)^r\rightarrow \hat B\otimes_k \det.
    \]
    As $\hat B\otimes_k kG$ is a projective object in the category of $(G,B)$-modules,
    $\tilde h$ factors through the surjection
    \[
    \hat B\otimes_k P_\nu\rightarrow \hat B\otimes_k \det.
    \]
    Returning to the category $\Ref\hat A$, $h$ factors through $\hat M_\nu
    =(\hat B\otimes_{\hat A} P_\nu)^G\rightarrow \omega_{\hat A}$.
    So this map must be surjective, and {\bf 2} follows.
  \end{proof}
    
  \begin{corollary}
    Assume that $p$ divides $\abs{G}$.
    If $s(\omega_{\hat A})>0$, then $\det^{-1}_V$ is not a direct summand of $B$.
  \end{corollary}

  \begin{proof}
    Being a one-dimensional representation, $\det^{-1}_V$ is not projective by assumption.
    Thus the result follows from {\bf 1$\Rightarrow$4} of the theorem.
  \end{proof}

  \begin{lemma}
    Let $M$ and $N$ be in $\Ref(G,B)$.
    There is a natural isomorphism
    \[
    \gamma:\Hom_A(M^G,N^G)\rightarrow \Hom_B(M,N)^G.
    \]
  \end{lemma}

  \begin{proof}
    This is simply because $\gamma=(B\otimes_A?)^{**}:\Ref(A)\rightarrow \Ref(G,B)$ is
    an equivalence, and $\Hom_B(M,N)^G = \Hom_{G,B}(M,N)$.
  \end{proof}

  \begin{theorem}
    $A$ is $F$-rational if and only if the following three conditions hold.
    \begin{enumerate}
    \item[\bf 1] $A$ is Cohen--Macaulay.
    \item[\bf 2] $H^1(G,B)=0$.
    \item[\bf 3] $(B\otimes_k (I/k))^G$ is a maximal Cohen--Macaulay $A$-module,
      where $I$ is the injective hull of $k$.
    \end{enumerate}
  \end{theorem}

  \begin{proof}
    If the order $|G|$ of $G$ is not divisible by $p$, then $A$ is $F$-rational, and
    the three conditions hold.
    So we may assume that $|G|$ is divisible by $p$.

    Assume that $A$ is $F$-rational.
    Then $A$ is Cohen--Macaulay.
    As $s(\omega_{\hat A})>0$, we have that $H^1(G,B\otimes_k \rad P_\nu)=0$,
    and
    \begin{equation}\label{F-rational.eq}
    0\rightarrow
    (B\otimes\rad P_\nu)^G
    \rightarrow
    (B\otimes P_\nu)^G
    \rightarrow
    (B\otimes\det_V)^G
    \rightarrow
    0
    \end{equation}
    is exact.
    As $M_\nu=(B\otimes P_\nu)^G$ is a direct summand of $B=M_{kG}=(B\otimes kG)^G$,
    it is a maximal Cohen--Macaulay module.
    As $(B\otimes\det)^G=\omega_A$, it is also a maximal Cohen--Macaulay module.
    So the canonical dual of the exact sequence (\ref{F-rational.eq}) is still exact.
    As there is an identification
    \[
    \Hom_A((B\otimes_k ?)^G,\omega_A)=\Hom_B(B\otimes_k ?,B\otimes_k \det_V)^G
    =(B\otimes_k ?^*\otimes_k \det_V)^G,
    \]
    we get the exact sequence of maximal Cohen--Macaulay $A$-modules
    \begin{equation}\label{F-rational2.eq}
    0\rightarrow A\rightarrow (B\otimes_k P_\nu^*\otimes_k \det_V)^G
    \rightarrow (B\otimes_k (\rad P_\nu)^*\otimes_k \det_V)^G
    \rightarrow 0.
    \end{equation}
    As $(\rad P_\nu)^*\otimes\det_V\cong I/k$, $(B\otimes(I/k))^G$ is maximal
    Cohen--Macaulay.
    As $I$ is an injective $G$-module, $B\otimes_k I$ is so as a $G$-module,
    and hence $H^1(G,B\otimes_k I)=0$.
    By the long exact sequence of the $G$-cohomology, we get $H^1(G,B)=0$.

    The converse is similar.
    Dualizing (\ref{F-rational2.eq}), we have that (\ref{F-rational.eq}) is exact.
  \end{proof}

  \begin{corollary}
    If $A$ is $F$-rational, then $H^1(G,k)=0$.
  \end{corollary}
    
  \begin{proof}
    $k$ is a direct summand of $B$, and $H^1(G,B)=0$.
  \end{proof}

  \begin{example}
    If $p=2$ and $G=S_2$ or $S_3$, the symmetric groups, then $H^1(G,k)\neq 0$.
    So $A$ is not $F$-rational, provided $G$ does not have a pseudo-reflection.
  \end{example}

  \section{An example of $F$-rational ring of invariants which are not $F$-regular}
  \label{main-example.sec}

  \paragraph
  Let $p$ be an odd prime number, and $k$ an algebraically closed field of
  characteristic $p$.

  \paragraph
  Let us identify $\Map(\Bbb F_p,\Bbb F_p)^\times$ with the symmetric group $S_p$.
  We write $\Bbb F_p=\{0,1,\ldots,p-1\}$.
  Define
  \begin{eqnarray*}
  G & := & \{\phi\in S_p\mid \exists a\in\Bbb F_p^\times\;\exists b\in\Bbb F_p\;
  \forall x\in\Bbb F_p\;\phi(x)=ax+b\}\subset S_p;\\
  Q & := & \{\phi\in Q\mid \exists b\in\Bbb F_p\;
  \forall x\in\Bbb F_p\;\phi(x)=x+b\}\subset G;\\
  \Gamma & := & \{\phi\in S_p\mid \exists a\in\Bbb F_p^\times\;
  \forall x\in\Bbb F_p\;\phi(x)=ax\}\subset G.
  \end{eqnarray*}
  $G$ is a subgroup of $S_p$, $Q$ is a normal subgroup of $G$, and $\Gamma$ is a
  subgroup of $G$ such that $G=Q\rtimes \Gamma$.
  Note that $Q$ is cyclic of order $p$.
  $\Gamma$ is cyclic of order $p-1$.
  So $G$ is of order $p(p-1)$.

  \paragraph
  Let $\alpha$ be a primitive element of $\Bbb F_p$ (that is, a generator of the
  cyclic group $\Bbb F_p^\times$), and let $\tau\in\Gamma$ be the element
  given by $\tau(x)=\alpha x$.
  The only involution of $\Gamma$ is $\tau^{(p-1)/2}$, the multiplication by $-1$.
  As a permutation, it is
  \[
  (1\;(p-1))(2\;(p-2))\cdots((p-1)/2\;(p+1)/2),
  \]
  which is a transposition if and
  only if $p=3$.
  As $\Gamma$ contains a Sylow $2$-subgroup, a transposition of $G$, if any, is
  conjugate to an element of $\Gamma$, and it must be a transposition again.
  It follows that $G$ has a transposition if and only if $p=3$.

  \paragraph
  Now let $G\subset S_p$ act on $P=k^p=\langle w_0,w_1,\ldots,w_{p-1}\rangle$
  by the permutation action, that is, $\phi w_i=w_{\phi(i)}$ for $\phi\in G$ and
  $i\in\Bbb F_p$.
  $g\in G\subset \GL(P)$ is a pseudo-reflection if and only if it is a transposition.
  So $G$ has a pseudo-reflection if and only if $p=3$.
  
  Let $r\geq 1$, and set $V=P^{\oplus r}$.
  $G\subset \GL(V)$ has a pseudo-reflection if and only if $p=3$ and $r=1$.

\paragraph
  Let $S=\Sym P$.

\begin{lemma}\label{S-decomposition.lem}
  Let $M$ be any finitely generated non-projective
  indecomposable $G$-summand of $S$.
  Then $M\cong k$.
\end{lemma}

\begin{proof}
  Let $\Omega=\{w^\lambda=w_0^{\lambda_0}\cdots w_{p-1}^{\lambda_{p-1}}\mid \lambda
  =(\lambda_0,\ldots,\lambda_{p-1})\in\Bbb Z_{\geq 0}^p\}$ be the set of monomials
  of $S$.
  $G$ acts on the set $\Omega$.
  Let $\Theta$ be the set of orbits of this action of $G$ on $\Omega$.
  Let $G w^\lambda\in\Theta$.

  If $\lambda=(r,r,\ldots,r)$ for some $r\geq 0$, then $G w^\lambda=\{w^\lambda\}$,
  and hence $(kG)w^\lambda\cong k$.

  Otherwise, $Q$ does not have a fixed point on the action on $G w^\lambda$.
  As the order of $Q$ is $p$, $Q$ acts freely on $Gw^\lambda$.
  Hence $(kG)w^\lambda$ is $kQ$-free.

  Since the order of $G/Q\cong\Gamma$ is $p-1$, the Lyndon--Hochschild--Serre spectral
  sequence collapses, and
  we have $H^i(G,M)\cong H^i(Q,M)^\Gamma$ for any $G$-module $M$.
  So a $Q$-injective (or equivalently, $Q$-projective) $G$-module is $G$-injective (or
  equivalently, $G$-projective).

  As we have $S=\bigoplus_{\theta\in\Theta}k\theta$ as a $G$-module,
  $S$ is a direct sum of $G$-projective modules and copies of $k$.
  Using Krull-Schmidt theorem, it is easy to see that $M\cong k$.
\end{proof}

\begin{lemma}\label{projective-tensor.lem}
  Let $U$ and $W$ be $G$-modules.
  \begin{enumerate}
  \item[\bf 1] $kG\otimes_k W\cong kG\otimes_k W'$, where $W'$ is the $k$-vector
    space $W$ with the trivial $G$-action.
  \item[\bf 2] If $U$ is $G$-projective, then $U\otimes_k W$ is $G$-projective.
  \end{enumerate}
\end{lemma}

\begin{proof}
  {\bf 1}.
  $g\otimes w \mapsto g\otimes g^{-1}w$ gives such an isomorphism.

  {\bf 2} follows from {\bf 1}.
\end{proof}

\paragraph
Let $B:=\Sym V=\Sym P^{\oplus r}\cong S^{\otimes r}$.

\begin{lemma}\label{B-decomposition.lem}
  Let $M$ be any finitely generated non-projective
  indecomposable $G$-summand of $B$.
  Then $M\cong k$.
\end{lemma}
  
\begin{proof}  
Follows immediately from Lemma~\ref{S-decomposition.lem} and
Lemma~\ref{projective-tensor.lem}.
\end{proof}

\begin{lemma}\label{det_V.lem}
  Let $k_-$ denote the sign representation.
  Then   $\det_V\cong k_-$ if $r$ is odd, and $\det_V\cong k$ if $r$ is even.
  $k_-$ is not isomorphic to $k$.
\end{lemma}

\begin{proof}
  As the determinant of a sign matrix is the signature of the permutation,
  $\det_P\cong k_-$.
  Hence $\det_V \cong (\det_P)^{\otimes r}\cong (k_-)^{\otimes r}$,
  and we get the desired result.
  The last assertion is clear, since $\tau=(x\mapsto \alpha x)\in\Gamma$ is a
  cyclic permutation of order $p-1$, and is an odd permutation.
\end{proof}

\begin{theorem}\label{kemper.thm}
  We have
  \[
  \depth_A =\min\{rp,2(p-1)+r\}.
  \]
  Hence $A$ is Cohen--Macaulay if and only if $r\leq 2$.
\end{theorem}

\begin{proof}
  This is an immediate consequence of \cite[(3.3)]{Kemper}.
\end{proof}

\begin{theorem}\label{main-example.thm}
  Let $p$, $r$, $G$, $P$, $V=P^{\oplus r}$, $B=\Sym V$ be as above, and $A:=B^G$.
  Then
  \begin{enumerate}
  \item[\bf 1] $G$ is a finite subgroup of $\GL(V)$ of order $p(p-1)$.
  \item[\bf 2] $G\subset \GL(V)$ has a pseudo-reflection if and only if $p=3$ and $r=1$.
    If so, $G=S_3$ is the symmetric group acting regularly on $B=k[w_0,w_1,w_2]$ by
    permutations on $w_0,w_1,w_2$.
    The ring of invariants $A$ is the polynomial ring.
    Otherwise, $A$ is not weakly $F$-regular.
  \item[\bf 3] If $p\geq 5$ and $r=1$, then $A$ is $F$-rational, but not weakly
    $F$-regular.
  \item[\bf 4] If $r=2$, then $A$ is Gorenstein, but not $F$-rational.
  \item[\bf 5] If $r\geq 3$ and $r$ is odd, then $s(\omega_{\hat A})>0$ but
    $A$ is not Cohen--Macaulay.
  \item[\bf 6] If $r\geq 4$ and even, then $A$ is quasi-Gorenstein, but not
    Cohen--Macaulay.
  \end{enumerate}
\end{theorem}

\begin{proof}
  We have already seen {\bf 1} and the first statement of {\bf 2}.
  If $p=3$ and $r=1$, then
  $G\subset S_3$ has order $6$, and $G=S_3$.
  So $A$ is the polynomial ring generated by the symmetric polynomials.
  Otherwise, as $G$ does not have a pseudo-reflection and the order $\abs{G}$ of $G$
  is divisible by $p$, $A$ is not weakly $F$-regular, see \cite{Broer}, \cite{Yasuda},
  and \cite[(5.8)]{SH}.

  The only non-projective finitely generated indecomposable $G$-summand of $B$ is
  $k$ by Lemma~\ref{B-decomposition.lem}, and $\det_V^{-1}\subset k$ if and only
  if $r$ is even by Lemma~\ref{det_V.lem}.
  Hence we have that $s(\omega_{\hat A})>0$ if
  and only if $r$ is odd by Theorem~\ref{main.thm}.

  {\bf 3}. $A$ is not weakly $F$-regular by {\bf 2}.
  As $r=1$ is odd, $s(\omega_{\hat A})>0$.
  On the other hand, $A$ is Cohen--Macaulay by Theorem~\ref{kemper.thm}.
  Hence $A$ is $F$-rational by Theorem~\ref{Sannai.thm}.

  {\bf 4}.
  By Theorem~\ref{kemper.thm}, $A$ is Cohen--Macaulay.
  On the other hand, by Lemma~\ref{det_V.lem}, $\det_V\cong k$, and hence
  $\omega_A\cong (B\otimes_k \det_V)^G\cong B^G\cong A$ by Lemma~\ref{canonical.lem}.
  So $A$ is Gorenstein.
  As $A$ is Gorenstein but not weakly $F$-regular, it is not $F$-rational
  by \cite[(4.7)]{HH2}.
  
  {\bf 5} and {\bf 6} are easy.
\end{proof}

\end{document}